\documentclass[a4paper,12pt]{amsart}
\usepackage{amsmath}
\usepackage{amsthm}
\usepackage{amssymb}
\usepackage{fullpage}
\usepackage{bbm}

\newtheorem{corollary}{Corollary}
\newtheorem{theorem}{Theorem}
\newtheorem*{rem}{Remark}

\newtheorem{prop}{Proposition}
\newtheorem{lem}{Lemma}
\newtheorem{lemma}{Lemma}

\newcommand{\assign}{:=}
\newcommand{\mathd}{\mathrm{d}}
\newcommand{\mathi}{\mathrm{i}}

\newenvironment{Proof}{\noindent\textbf{Proof\ }}{\hspace*{\fill}$\Box$\medskip}

\newcommand{\tRe}{\textup{Re }}
\newcommand{\tIm}{\textup{Im }}
\newcommand{\bfrac}[2]{\left(\frac{#1}{#2}\right)}

\begin{document}
\title{The Riemann zeta function on vertical arithmetic progressions}
\author{Xiannan Li}
\address{Department of Mathematics \\ University of Illinois at Urbana-Champaign \\
1409 W. Green Street \\ Urbana, IL 61801 USA}
\email{xiannan@illinois.edu}
\author{Maksym Radziwi\l\l}
\address{Department of Mathematics \\ Stanford University \\
450 Serra Mall, Bldg. 380\\ Stanford, CA 94305-2125}
\email{maksym@stanford.edu}

\thanks{The second author is partially supported by a NSERC PGS-D award.}

\subjclass[2010]{Primary: 11M06, Secondary: 11M26}

\begin{abstract} We show that the twisted second moments of the Riemann zeta function averaged over the arithmetic progression
$\tfrac 12 + i(an + b)$ with $a > 0$, $b$ real, exhibits a remarkable correspondance with the analogous continuous average and derive several consequences.  For example, motivated by the linear independence conjecture,
we show at least
one third of the elements in the arithmetic progression $a n + b$ 
are not the ordinates of some zero of $\zeta(s)$
lying on the critical line. This improves on earlier work of
Martin and Ng. We then complement this result by producing
large and small values of $\zeta(s)$ on arithmetic progressions which are of the same quality as the best $\Omega$ results currently known for $\zeta(\tfrac 12 + it)$ with $t$ real.
\end{abstract}

\maketitle

\section{Introduction}
In this paper, we study the behavior of the Riemann zeta function $\zeta(s)$ in vertical arithmetic progressions on the critical line.  To be more precise, fix real numbers $\alpha>0$ and $\beta$.  We are interested in the distribution of values of $\zeta(1/2 + i(\alpha \ell+\beta))$ as $\ell$ ranges over the integers in some large dyadic interval $[T, 2T]$.  Here are some specific questions of interest:
\begin{enumerate}
 \item How does the mean square $\sum_{\ell \in [T, 2T]} |\zeta(\tfrac 12 + i\ell)|^2$ compare to $\int_T^{2T} |\zeta(\tfrac 12 +it)|^2 dt$?  
\item Does the mean square of $\zeta(s)$ distinguish arithmetic sequences? That is, does $\sum_{\ell \in [T, 2T]} |\zeta(1/2 + i (\alpha \ell+\beta))|^2$ depend on $\alpha$ and $\beta$? 
\item What about the case $\sum_{\ell \in [T, 2T]} |\zeta (\tfrac 12 + i (\alpha \ell+\beta))B(\tfrac 12 + i(\alpha \ell + \beta))|^2$, where $B(s)$ is an arbitary Dirichlet polynomial?  In the special when $B(s)$ is a mollifier, the continuous average of $\zeta(\tfrac 12 + it) B(\tfrac 12 +it)$ has been shown to be close to $1$.  Does $B(s)$ still act the same way when restricted to the discrete sequence $\tfrac 12 + i(\alpha \ell + \beta)$?
\end{enumerate}

For most - but not all - values of $\alpha$ and $\beta$
our results suggest that the average behavior of
$\zeta(\tfrac 12 + i(\alpha \ell + \beta))$ is similar to that of a unitary family such as $L(\tfrac 12 ; \chi)$.  

Besides being of independent interest the above three questions are motivated
by the linear independence conjecture, which we approach through
two simpler questions:
\begin{enumerate}
\item Can $\zeta(s)$ vanish at many (or most) of the points $\tfrac 12 + 
i(\alpha \ell + \beta)$?
\item Can $\zeta(s)$ be extremely large or small at a point of 
the form $\tfrac 12 + i (\alpha \ell + \beta)$? Are the extreme values
at $\tfrac 12 + i(\alpha \ell + \beta)$ comparable to those of 
$\zeta(\tfrac 12 + it)$ with $t \in [T;2T]$?
%
%
\end{enumerate}
We begin with some mean square results.

\subsection{Mean value estimates}
The distribution of values of $\zeta(s)$ on the critical line has been studied extensively by numerous authors and in particular the moments of $\zeta(s)$ have received much attention.  
Consider a Dirichlet polynomial $B(s)$ with,
\begin{equation}\label{eqn:B(s)}
B(s) = \sum_{n \leq T^{\theta}} \frac{b(n)}{n^s} , \text{ and } b(n) \ll d_{A}(n)
\end{equation}
for some fixed, but arbitrary $A > 0$. Throughout we will assume that the
coefficients $b(n)$ are real.
\begin{theorem} \label{mirage2}
Let $B(s)$ be as above. 
Let $\phi(\cdot)$ be a smooth compactly supported function, with support
in $[1,2]$.  If $\theta < \tfrac 12$, then 
as $T \rightarrow \infty$.
$$
\sum_{\ell} |\zeta(\tfrac 12 + i\ell)B(\tfrac 12 + i\ell)|^2 \cdot
\phi \bigg ( \frac{\ell}{T} \bigg ) = \int_{\mathbbm{R}} |\zeta(\tfrac 12
 + it)B(\tfrac 12 + it)|^2 \cdot \phi \bigg ( \frac{t}{T} \bigg )dt 
+ O_{A}(T (\log T)^{-A})).
$$
\end{theorem}
Since $\zeta(s)B(s)$ oscillates on a scale of $2\pi / \log T$ it is
interesting that we can reconstruct accurately the continuous average of $\zeta(s)B(s)$ only by sampling at the integers.  The reader may be amused by examining the same statement for $\sin x$ or $\sin (\log (|x|+1) x)$, which will be equivalent to the equidistribution of certain sequences modulo $1$.

Theorem \ref{mirage2} depends on the fact that we are summing over the integers, and specifically on the fact that the sequence $e^{2\pi \ell}$ cannot be well approximated by rational numbers.  To amplify this dependence, let us consider 
the second moment of $\zeta(s)$ averaged over an arithmetic progression $\alpha n + \beta$,
with arbitrary $\alpha > 0$ and $\beta$.  
In this context, our result will depend on the diophantine properties of $e^{2\pi \ell / \alpha}$. 
Let 
$$
\delta(\alpha,\beta) = \begin{cases}
0 & \text{ if } e^{2\pi\ell / \alpha} \text{ is irrational for all } \ell > 0 \\
\frac{2\cos(\beta \log (m/n)) \sqrt{m n} - 2}{m n + 1 - 2 \sqrt{m n} \cos(\beta \log(m / n))} & \text{ if } e^{2\pi\ell/\alpha} \text{ is rational for some } \ell > 0
\end{cases}
$$
with $m/n \neq 1$ denoting the smallest reduced fraction 
having a representation in the form $e^{2\pi\ell/\alpha}$ for some $\ell > 0$.
Then we have the following asymptotic result for the second moment of
the Riemann zeta function.
\begin{theorem} \label{thm:mirage}
Let $\phi(\cdot)$ be a smooth compactly supported function, with support
in $[1,2]$. Let $\alpha > 0$, $\beta$ be real numbers.  Then, as $T \rightarrow \infty$,
$$
\sum_{\ell} |\zeta(\tfrac 12 + i(\alpha\ell + \beta))|^2 \cdot \phi \bfrac{\ell}{T}
= \int_{\mathbbm{R}} |\zeta(\tfrac 12 + i(\alpha t + \beta))|^2 \cdot
\phi \bfrac{t}{T} dt \cdot ( 1 + \delta(\alpha,\beta) +  o(1))
$$
\end{theorem}
In the above, $o(1)$ denotes a quantity tending to $0$ as $T$ grows, which depends on the diophantine properties of $\alpha$ and $\beta$.  Our methods allow us to prove an analogous result for the second moment of
$\zeta(s)$ twisted by a Dirichlet polynomial over an arbitrary
vertical arithmetic progression.  See Proposition \ref{prop:aver} for more details.

In contrast to Theorem \ref{thm:mirage}, the dependence on the diophantine properties of $\alpha$ and $\beta$ is nullified when $B$ is a mollifier.  To be precise, let $\phi(\cdot)$ be a smooth compactly supported function, with support
in $[1,2]$, and define
$$
M_{\theta}(s) := \sum_{n \leq T^{\theta}} \frac{\mu(n)}{n^s}
\cdot \bigg ( 1 - \frac{\log n}{\log T^{\theta}} \bigg ).
$$Then we have the following Theorem.

\begin{theorem}\label{thm:mollifiedmirage}
Let the mollified second moment be defined as
\begin{equation}\label{eqn:J}
\mathcal{J} := \sum_{\ell} |\zeta(\tfrac 12 + i(\alpha\ell + \beta))
 M_{\theta} (\tfrac 12+ i(\alpha \ell+\beta))|^2 \phi
\bfrac{\ell}{T}. 
\end{equation}
Let $0 < \theta < \tfrac 12$ and $a > 0$ and $b$ be real numbers.  Then, 
$$
\mathcal{J} =
\int_{\mathbbm{R}} \big | (\zeta \cdot M_{\theta}) (\tfrac 12 + i (\alpha t + \beta))|^{2}
\cdot \phi \bfrac{t}{T} dt + O \bigg (\frac{T}{(\log T)^{1 - \varepsilon}} \bigg )
$$
\end{theorem}

The lack of dependence on the diophantine properties of $\alpha$ and $\beta$ in Theorem \ref{thm:mollifiedmirage} gives the non-vanishing proportion of $\tfrac 13$ in Theorem \ref{thm:nonvanishing} below.  

\subsection{Non-vanishing results}
One of the fundamental problems in analytic number theory is determination of the location of the zeros of $L$-functions.  Here, one deep conjecture about the vertical distribution of zeros of $\zeta(s)$ is the Linear Independence Conjecture (LI), which states that the ordinates of non-trivial zeros of $\zeta(s)$ are linearly independent over $\mathbb{Q}$.  In general, it is believed that the zeros of $L$-functions do not satisfy any algebraic relations, but rather appear to be ``random'' transcendental numbers.   Classically, Ingham \cite{Ingham} linked the linear independence conjecture for the Riemann zeta-function with the
oscilations of $M(x) = \sum_{n \leq x} \mu(n)$, in particular offering a 
conditional disproof of Merten's conjecture that $|M(x)| \leq \sqrt{x}$
for all $x$ large enough.  There are a number of connections between LI and the distribution of primes.  For instance, Rubinstein and Sarnak \cite{RubinsteinSarnak} showed a connection between LI for Dirichlet $L$-functions and prime number races, and this has appeared in the work of many subsequent authors. 

LI appears to be far out of reach of current technology.  However, it implies easier conjectures which may be more tractable.  One of these is that the vertical ordinates of nontrivial zeros of $\zeta(s)$ should not lie in an arithmetic progression.  To be more precise, for fixed 
$\alpha > 0$, $\beta \in \mathbbm{R}$, 
let
\begin{equation*}
P_{\alpha,\beta}(T) = \frac{1}{T} \cdot \text{Card} \{T \leq \ell
\leq 2T: \zeta(\tfrac 12 + i(\alpha\ell+\beta)) \neq 0 \}. 
\end{equation*}
Then what kind of lower bounds can we prove for $P_{\alpha,\beta}(T)$ for large $T$?  
Recently, improving on the work of numerous earlier authors, Martin and Ng \cite{MartinNg} showed that 
$P_{\alpha,\beta}(T) \gg_{\alpha,\beta} (\log T)^{-1}$ 
which misses the truth by a factor of $\log T$.  
In this paper, we prove the following improvement.
\begin{theorem} \label{thm:nonvanishing}
Let $\alpha > 0$ and $\beta$ be real. Then, as $T \rightarrow \infty$, 
$$
P_{\alpha,\beta}(T) \geq \frac 13 + o(1). 
$$
\end{theorem}
The proof of Theorem \ref{thm:nonvanishing} leads easily to the result below.
\begin{corollary}
Let $\alpha > 0$ and $\beta$ be real. Then, as $T \rightarrow \infty$,
$$
|\zeta(\tfrac 12 + i(\alpha\ell + \beta)| \geq \varepsilon (\log \ell)^{-1/2}
$$
for more than $(\tfrac 13 - C \varepsilon) T$ integers $T \leq \ell \leq 2T$, with
$C$ an absolute constant. 
\end{corollary}
Theorem \ref{thm:nonvanishing} is proven by understanding both a mollifed discrete second moment (see Theorem \ref{thm:mollifiedmirage}) and a mollified discrete first moment.  Our methods extend without modification to prove the analogous result for Dirichlet $L$-functions. The constants $\tfrac 13$ represents the limits of the current technology - see for example \cite{IwaniecSarnak} for the case of non-vanishing of Dirichlet $L$-functions at the critical point.

Of course, we expect that $P_{\alpha, \beta}(T) = 1+O(T^{-1})$. Assuming the Riemann Hypothesis (RH),
Ford, Soundararajan and Zaharescu \cite{FordSoundararajanZaharescu} 
showed $P_{\alpha,\beta}(T) \geq \tfrac 12 + o(1)$ as $T \rightarrow \infty$. 
Assuming RH and Montgomery's Pair Correlation Conjecture
they've showed \cite{FordSoundararajanZaharescu} that $P_{\alpha,\beta}(T) \geq 1 - o(1)$ as $T \rightarrow \infty$.
Assuming a very strong hypothesis on the distribution of primes in short intervals, it is possible to show that $P_{\alpha, \beta}(T) = 1 - O(T^{-\delta})$ for some $\delta > 0$.

Note that the rigid structure of the arithmetic progression is important.  
Since there is a zero of 
$\zeta(s)$ in every interval of size essentially $(\log \log \log T)^{-1}$ in $[T, 2T]$  (see \cite{Littlewood}) minor perturbations of the arithmetic progression renders our 
result false.  

\subsection{Large and small values}
We now complement Theorem \ref{thm:nonvanishing} by exhibiting large and small values of $\zeta(s)$ at discrete points $\tfrac 12 + i(\alpha\ell + \beta)$ using Soundararajan's resonance method \cite{Soundararajan}. 
\begin{theorem}\label{thm:largesmall}
Let $\alpha > 0$ and $\beta$ be real. Then, for infinitely many $\ell > 0$, 
$$
|\zeta(\tfrac 12 + i(\alpha\ell+\beta))| \gg  \exp \bigg ( ( 1 + o(1)) \sqrt{\frac{\log \ell}{6 \log\log \ell}} \bigg )
$$
and for infinitely many $\ell$,
$$
|\zeta(\tfrac 12 + i(\alpha\ell+\beta))| \ll  \exp \bigg ( - (1 + o(1)) \sqrt{\frac{\log \ell}{6 \log\log \ell}} \bigg ).
$$
\end{theorem}
The $o(1)$ in
this result is independent of the diophantine properties of $\alpha$ and $\beta$. 
Since we expect $\zeta(\tfrac 12 + i(\alpha\ell+\beta)) \neq 0$ for essentially all $\ell$, it is interesting to produce values of $\ell$ at which $\zeta(\tfrac 12 + i(\alpha\ell+\beta))$ is extremely small. Furthermore, the large values of $\zeta(\tfrac 12 + i(\alpha\ell + \beta))$ over a discrete set of points above are almost of the same quality as the best results for large values of $\zeta(\tfrac 12 + it)$ with $t$ real.  In the latter case, the best result is due to Soundararajan \cite{Soundararajan}.  We have not tried to optimize in Theorem \ref{thm:largesmall} and perhaps the same methods might lead to the constant $1$ rather than $1 / \sqrt{6}$.

\subsection{Technical propositions}
The proofs of our Theorems rests on a technical Proposition, and its variant, which may be of independent interest.  With $B(s)$ defined as in (\ref{eqn:B(s)}), consider the difference between the discrete average and the continuous average,
$$
\mathcal{E} := \sum_{\ell} |(\zeta \cdot B)(\tfrac 12 + i(\alpha\ell + \beta))|^2
\phi \bigg (\frac{\ell}{T} \bigg ) - \int_{\mathbbm{R}} |(\zeta \cdot B)(\tfrac 12
+ i (\alpha t+ \beta))|^2 \phi \bigg ( \frac{t}{T} \bigg ) dt.
$$
Proposition \ref{prop:aver} below shows that understanding $\mathcal{E}$ boils down to understanding the behavior of sums of the form 
\begin{equation}\label{eqn:F}
F (a_{\ell},b_{\ell}, t) := \sum_{r \geqslant 1}
\frac{1}{r} \sum_{h, k \leqslant T^{\theta}} b (k) b (h)
\sum_{\substack{
    m, n \geqslant 1\\
    mk = a_{\ell} r\\
    nh = b_{\ell} r
}} W \bigg ( \frac{2 \pi mn}{\alpha t + \beta} \bigg )
\end{equation}
where $W(x)$ is a smooth function  defined as
$$
W(x) := \frac{1}{2\pi i} \int_{(\varepsilon)} x^{-w} \cdot G(w) \frac{dw}{w}
$$
with $G(w)$ an entire function of rapid decay along vertical lines
$G(x + iy) \ll_{x,A} |y|^{-A}$, such that
$G(w) = G(-w)$, $G(0) = 1$, and satisfying $G(\bar w) = \overline{G(w)}$ (to make $W(x)$ real valued for $x$ real). For example we can take $G(w) = e^{w^2}$.
Notice that $W(x) \ll 1$ for $x \leq 1$ and $W(x) \ll_A x^{-A}$ for $x > 1$. 

Of course, the expression in $\ref{eqn:F}$ should not depend on the choice of $W$.  In fact, $F(a_l, b_l, t)$ can also be written as
\begin{equation} \label{form}
\sum_{m,n \leq T^{\theta}} \frac{b(m)b(n)}{m n}
\cdot (m a_{\ell}, n b_{\ell}) \cdot 
 \mathcal{H} \bigg ( (\alpha t + \beta) \cdot 
\frac{(ma_{\ell},nb_{\ell})^2}{2\pi m a_{\ell} n b_{\ell}} \bigg )
\end{equation}
where $\mathcal{H}(x)$ is a smooth function such that,
$$
\mathcal{H}(x) = \begin{cases}
\tfrac 12 \cdot \log x + \gamma + O_A(x^{-A}) & \text{ if } x \gg 1 \\
O_A(x^A) & \text{ if } x \ll 1
\end{cases}
$$
As seen in a theorem of Balasubramanian, Conrey and Heath-Brown
\cite{BalasubramanianConreyHeathBrown}
the continuous $t$ average over $T \leq t \leq 2T$ 
of $|\zeta(\tfrac 12 + it)B(\tfrac 12 + it)|^2$ gives rise to (\ref{form})
with $a_{\ell} = 1 = b_{\ell}$. 
For technical reasons it is more convenient for us to work with
the smooth version (\ref{eqn:F}).

\begin{prop}\label{prop:aver}
Let $0<\theta <1/2$.  For each $\ell > 0 $, let $(a_{\ell},b_{\ell})$ denote (if it exists) the unique tuple of co-prime integers such that $a_{\ell}b_{\ell} > 1$, $b_{\ell} < T^{1/2 - \varepsilon} e^{-\pi \ell / \alpha}$ and
     \begin{equation} \label{condition4}
       \bigg | \frac{a_{\ell}}{b_{\ell}} - e^{2\pi \ell / \alpha} \bigg | \leq
       \frac{e^{2\pi \ell / \alpha}}{T^{1 - \epsilon}}.
     \end{equation} 
If such a pair $(a_{\ell},b_{\ell})$ exists, then let
$$H(\ell) 
= \frac{(a_{\ell}/b_{\ell})^{i\beta}}{\sqrt{a_{\ell}b_{\ell}}} \int_{- \infty}^{\infty} \phi \bigg (
\frac{t}{T} \bigg )
  \cdot \exp \bigg ( - 2\pi i t \bigg (
  \frac{\alpha \log  \frac{a_{\ell}}{b_{\ell}}}{2\pi} - \ell \bigg ) \bigg )
  \cdot F(a_{\ell},b_{\ell}, t) d t,
$$and otherwise set $H(\ell)=0$.
Then,
$$
\mathcal{E} =  4 \tRe   \sum_{\ell > 0}H(\ell) + O(T^{1 - \varepsilon}).
$$
\end{prop}
More generally we can consider
$$
\mathcal{E}' = \sum_{\ell} | B(\tfrac 12 + i(\alpha\ell + \beta))|^2
\phi \bigg (\frac{\ell}{T} \bigg ) - \int_{\mathbbm{R}} |B(\tfrac 12
+ i (\alpha t + \beta))|^2 \phi \bigg ( \frac{t}{T} \bigg ) dt.
$$
In this case our results depend on
$$
F'(a_{\ell},b_{\ell}) := \sum_{r \geq 1} \frac{b(a_{\ell} r)b(b_{\ell} r)}{r}
$$
where we adopted the convention that $b(n) = 0$ for $n > T^{\theta}$. Then the analogue
of Proposition 1 is stated below.
\begin{prop}
Let $0 < \theta < 1$. For each $\ell > 0$ let $(a_{\ell}, b_{\ell})$ denote
(if it exists) the unique tuple of co-prime integers such that $a_{\ell}b_{\ell} > 1$, $b_{\ell} < T^{1/2 - \varepsilon} e^{-\pi \ell / \alpha}$ and
     \begin{equation} \label{condition4}
       \bigg | \frac{a_{\ell}}{b_{\ell}} - e^{2\pi \ell / \alpha} \bigg | \leq
       \frac{e^{2\pi \ell / \alpha}}{T^{1 - \epsilon}}.
     \end{equation} 
Then,
$$
\mathcal{E}' = 2 \Re \sum_{\ell > 0} 
\frac{(a_{\ell}/b_{\ell})^{i\beta}}{\sqrt{a_{\ell}b_{\ell}}} \cdot \hat{\phi} \bigg ( \frac{\alpha \log \frac{a_{\ell}}{b_{\ell}}}{2\pi}
- \ell \bigg ) F^{'}(a_{\ell},b_{\ell}) + O(T^{1 - \varepsilon})
$$
where in the summation over $\ell$ we omit the terms for which the
pair $(a_{\ell},b_{\ell})$ does not exist.
\end{prop}
The proof of Proposition 2 is very similar (in fact easier!) than that of Proposition 1, and for this reason we omit it. 

One can ask about the typical distribution of $\log \zeta(\tfrac 12 + i(\alpha \ell + \beta))$ . This question is out of reach if we focus on the real part of
$\log \zeta(s)$ since we cannot even guarantee that almost all $\tfrac 12 + i(\alpha\ell + \beta)$ are not zeros of the Riemann zeta-function. On the Riemann Hypothesis, using Proposition 2 and Selberg's methods, one can prove a central limit theorem for
$S(\alpha \ell + \beta)$ with $T \leq \ell \leq 2T$.  We will not pursue this application here.

We deduce Theorems \ref{mirage2} and \ref{thm:mirage} from Proposition \ref{prop:aver} in Section 2.  We then prove Theorem \ref{thm:mollifiedmirage} in Section 3, complete the proof of Theorem \ref{thm:nonvanishing} in Section 4, and prove Theorem \ref{thm:largesmall} in Section 5.  Finally, we prove Proposition \ref{prop:aver} in Section 6. 

\section{Proof of Theorems \ref{mirage2} and \ref{thm:mirage}}
\begin{proof}[Proof of Theorem \ref{mirage2}]
Set $\alpha = 1$ and $\beta = 0$. 
By Proposition \ref{prop:aver} it is enough to show that 
$\mathcal{E} \ll T (\log T)^{-A}$. 
Since $W(x) \ll x^{-A}$ for $x > 1$ and $W(x) \ll 1$ for $x \leq 1$, we have, for $T \leq t \leq 2T$
\begin{equation*}
F(a_{\ell},b_{\ell},t)  \ll 1 + \sum_{r \geq 1} \frac{1}{r} \sum_{h, k \leq T^{\theta}}
|b(k)b(h)| \sum_{\substack{m ,n \leq T^{1 + \varepsilon}\\
  m k = a_{\ell} r \\ n h = b_{\ell} r}} 1
\ll \sum_{r \leq T^{2}} \frac{c(a_{\ell} r)c(b_{\ell} r)}{r} + 1
\end{equation*}
where
$
c(n) := \sum_{d | n} |b(d)| \ll  d_{A + 1}(n)
$.
Therefore $F(a_{\ell},b_{\ell},t) \ll (a_{\ell} b_{\ell})^{\varepsilon} T (\log T)^{B}$.
for some large $B > 0$.  
It thus follows by Proposition \ref{prop:aver}, that
$$
\mathcal{E} \ll T (\log T)^{B} \cdot \sum_{\ell > 0} (a_{\ell} b_{\ell})^{-1/2 + \varepsilon}
$$
Because of (\ref{condition4}) we have $a_{\ell} b_{\ell} \gg e^{2\pi \ell}$. Therefore
the $\ell$'s with $\ell \geq (\log \log T)^{1 + \varepsilon}$ 
contribute $\ll_{A} T (\log T)^{-A}$. We can therefore subsequently assume
that $\ell \ll (\log \log T)^{1 + \varepsilon}$. 
In order to control $a_{\ell}$ and $b_{\ell}$, when 
$\ell \leq (\log \log T)^{1 + \varepsilon}$ we appeal to a result of Waldschmidt
(see \cite{Waldschmidt}, p. 473),
\begin{equation} \label{wald}
\bigg | e^{\pi m} - \frac{p}{q} \bigg | \geq \exp\bigg ( -2^{72} \log(2m)
\log p \cdot \log \log p \bigg ).
\end{equation}
Therefore if condition (\ref{condition4}) is satisfied then
$
e^{2\pi \ell} T^{-1 + \varepsilon} \geq \exp ( - c (\log \ell)
\cdot (\log a_{\ell}) (\log \log a_{\ell}) )
$
Therefore, using that $\ell \leq (\log \log T)^{1 + \varepsilon}$ we
get $(\log a_{\ell}) \cdot (\log \log a_{\ell}) \gg \log T / (\log \log T)^{ \varepsilon}$, and hence
$\log a_{\ell} \gg \log T / (\log \log T)^{1 + \varepsilon}$. 
Notice also that (\ref{condition4})
implies that $a_{\ell} b_{\ell} \gg e^{2\pi \ell}$, so that
$\sum_{\ell > 0} (a_{\ell} b_{\ell})^{-\alpha} = O_{\alpha}(1)$ for any $\alpha > 0$.  
Combining these observations we find
$$
\sum_{0 < \ell < (\log\log T)^{1 + \varepsilon}} 
  (a_{\ell} b_{\ell})^{-1/2 + \varepsilon} \ll e^{ - c \log T / (\log \log T)^{1 + \varepsilon}}
\sum_{\ell > 0} (a_{\ell} b_{\ell})^{-1/4} \ll e^{-c \log T / (\log \log T)^{1 + \varepsilon}}.
$$
Thus $\mathcal{E} \ll_{A} T (\log T)^{-A}$ for any fixed $A > 0$, 
as desired.
\end{proof}
It is possible to generalize this theorem to other
progressions, for example to those for which $2\pi / \alpha$ is algebraic.
We refer the reader to \cite{Waldschmidt} for the necessary results in
diophantine approximation.

\begin{proof}[Proof of Theorem \ref{thm:mirage}]
Set $B(s) = 1$ in Proposition \ref{prop:aver}. Then, keeping notation
as in Proposition \ref{prop:aver}, we get
$$
\sum_{\ell} |\zeta(\tfrac 12 + i(\alpha \ell + \beta))|^2
\cdot \phi \bfrac{\ell}{T} = \int_{\mathbbm{R}}
|\zeta(\tfrac 12 + i(\alpha t + \beta))|^2 \cdot
\phi \bfrac{t}{T} dt + \mathcal{E}
$$
The main term is $\sim \hat{\phi}(0) T \log T$. 
It remains to understand $\mathcal{E}$. 
\\
\\
\noindent\textbf{First case}. First suppose that 
$e^{2\pi\ell/\alpha}$ is irrational for all
$\ell > 0$.
Since $b(k) = 1$ if $k = 1$ and $b(k) = 0$ otherwise
it is easy to see that $F(a_{\ell},b_{\ell},t) \ll T \log T$
uniformly in $a_{\ell},b_{\ell}$ and $T \leq t \leq 2T$. Thus,
$$
\mathcal{E} \ll T \log T \sum_{\ell > 0} (a_{\ell} b_{\ell})^{-1/2}
$$
It remains to show that $\sum_{\ell > 0} (a_{\ell} b_{\ell})^{-1/2} = o(1)$
as $T \rightarrow \infty$. Let $\varepsilon  > 0$ be given. 
Since $a_{\ell} b_{\ell} \gg e^{2\pi \ell / \alpha}$ we can find an $A$ such that
$
\sum_{\ell > A} (a_{\ell} b_{\ell})^{-1/2} \leq \varepsilon
$.
For the remaining integers $\ell \leq A$ notice that $e^{2\pi \ell / \alpha}$
is irrational for each $\ell \leq A$. Therefore for each $\ell \leq A$, 
\begin{equation} \label{condition5}
\bigg | \frac{a_{\ell}}{b_{\ell}} - e^{2\pi \ell / \alpha} \bigg | \leq
\frac{e^{2\pi \ell / \alpha}}{T^{1 - \varepsilon}}
\end{equation}
implies
that $a_{\ell} b_{\ell} \rightarrow \infty$. It follows that
$\sum_{\ell \leq A} (a_{\ell} b_{\ell})^{-1/2} \leq \varepsilon$ once $T$ is large enough.
We conclude that $\sum_{\ell > 0} (a_{\ell} b_{\ell})^{-1/2} = o(1)$, and hence
that $\mathcal{E} = o(T \log T)$ as desired.
\\
\\
\noindent
\textbf{Second case}. 
Now consider the case that
$e^{2\pi\ell_0 / \alpha}$ is rational for some $\ell_0$.
Write 
\begin{equation} \label{contradictus}
\alpha = \frac{2\pi \ell_0}{\log (m/n)}
\end{equation}
with co-prime $m$ and
$n$ and $|m|$ minimal. Let $k$ be the maximal positive integer
such that $m/n = (r/s)^k$ with $r, s$ co-prime. Then,
$$
\alpha = \frac{\ell_0}{k} \cdot \frac{2\pi}{\log(r/s)}.
$$Let $d = (\ell_0, k)$.  Note that $d=1$ since otherwise, we may replace $\ell_0$ by $\ell_0/d$ and $m$ and $n$ by $m^{1/d}$ and $n^{1/d}$ in 
(\ref{contradictus})
 which contradicts the minimality condition on $|m|$.

For each $\ell$ divisible by $\ell_0$
the integers $a_{\ell} = m^{\ell / \ell_0}$ and $b_{\ell} = n^{\ell / \ell_0}$
satisfy (\ref{condition5}) because
$e^{2\pi \ell / \alpha} = (m/n)^{\ell / \ell_0}$. 
For the remaining integers $\ell$ not divisible by $\ell_0$, $e^{2\pi \ell / \alpha} = (r/s)^{k \ell / \ell_0}$ is irrational, since $\ell_0|k\ell$ if and only if $\ell_0|\ell$.
We split $\mathcal{E}$ accordingly
$$
\mathcal{E} = 4 \tRe \sum_{\substack{\ell > 0 \\ \ell_0 | \ell}} H(\ell)
+ 4 \tRe \sum_{\substack{\ell > 0 \\ \ell_0 \nmid \ell}} H(\ell)
$$
The second sum is $o(T\log T)$ as can be seen by repeating 
the same argument as in the first case. 
As for the first sum, we find that for each $\ell$ divisible by $\ell_0$,
$$
H(\ell) = 2 \tRe \bigg ( \frac{(m/n)^{i\beta}}{\sqrt{m n}} \bigg )^{\ell / \ell_0}
\cdot \hat{\phi}(0) T \log T + O \bigg ( \frac{\ell \log m n}{(m n)^{\ell/2\ell_0}} \cdot T \bigg ).
$$
Therefore
\begin{align*}
\sum_{\substack{\ell > 0 \\ \ell_0 | \ell}} H(\ell) & = 2 \hat{\phi}(0) T \log T \cdot
\sum_{\ell > 0} \bigg ( \frac{(m/n)^{i\beta}}{\sqrt{m n}} \bigg )^{\ell} + O(T) \\
& = \hat{\phi} (0) T \log T \cdot  \frac{2\cos(\beta \log (m / n)) \sqrt{m n} - 2}
{m n + 1 - 2\sqrt{m n}\cos(\beta \log (m / n))} + O(T)
\end{align*}
giving the desired estimate for $\mathcal{E}$. 
\end{proof}

\section{Proof of Theorem \ref{thm:mollifiedmirage}}
Recall that in the notation of Proposition \ref{prop:aver},
$$
F (a_{\ell},b_{\ell}, t) := \sum_{r \geqslant 1}
\frac{1}{r} \sum_{h, k \leqslant T^{\theta}} b (k) b (h)
\sum_{\substack{
    m, n \geqslant 1\\
    mk = a_{\ell} r\\
    nh = b_{\ell} r
}} W \bigg ( \frac{2 \pi mn}{\alpha t + \beta} \bigg ) $$
The lemma below, provides a bound for $F$ when the coefficients $b(n)$
are the coefficients of the mollifiers $M_{\theta}(s)$, that is
$$
b(n) = \mu(n) \cdot \bigg ( 1 - \frac{\log n}{\log T^{\theta}} \bigg )
$$
and $b(n) = 0 $ for $n > T^{\theta}$. 
\begin{lem}\label{lemma:saving}
 For any $a_{\ell}, b_{\ell} \in \mathbb{N}$ with $(a_{\ell}, b_{\ell})=1$ and $a_{\ell}b_{\ell} > 1$, 
 uniformly in $T \leq t \leq 2T$, we have that
$$
F (a_{\ell},b_{\ell},t) \ll (a_{\ell} b_{\ell})^{\varepsilon} \cdot T (\log T)^{- 1 + \varepsilon}. 
$$
%
\end{lem}
\begin{proof}
 For notational ease, let $N =T^\theta$.  We first express the conditions in the sum above in terms of Mellin transforms.  To be specific
since
$$
W(x) = \frac{1}{2\pi} \int_{(\varepsilon)} x^{-w} G(w) \frac{dw}{w}
$$
with $G(w)$ rapidly decaying along vertical lines, and such that
$G(w) = G(-w)$, $G(0) = 1$,  we have
\begin{align*}
 & S =
\frac{1}{2\pi i} \int_{(2)} \sum_{m, n \geq 1} \sum_{h, k \leq N} b(h)b(k)\sum_{\substack{r \geq 1\\ nk = b_{\ell} r \\ mh = a_{\ell} r}}  \frac{1}{r} \bfrac{\alpha t + \beta}{2\pi mn}^w G(w) \frac{dw}{w} \\ 
& = \bfrac{1}{2\pi i}^3 \int_{(2)}\int_{(2)} \int_{(2)} \sum_{m, n \geq 1} \frac{1}{(mn)^w} \sum_{h, k} \frac{\mu(h)\mu(k)}{h^{z_1}k^{z_2}} 
\sum_{\substack{r \geq 1\\nk = b_{\ell} r\\mh = a_{\ell} r}} \frac{1}{r} \bigg ( \frac{\alpha t +\beta}{2\pi} \bigg )^w G(w) \frac{dw}{w} \frac{N^{z_1} dz_1}{\log N z_1^2}\frac{N^{z_2} dz_2}{\log N z_2^2}.
\end{align*}
The sum over $m, n, h, k$ and $r$ inside the integral may be factored into an Euler product as
\begin{align*}
 &\sum_{r\geq 1} \frac{1}{r} \left( \sum_{nk = b_{\ell} r} \frac{1}{n^w} \frac{\mu(k)}{k^{z_2}} \right)
\left( \sum_{mh = a_{\ell} r} \frac{1}{m^w} \frac{\mu(h)}{h^{z_1}} \right)\\
 & = \prod_{p} \left( 1+\frac{1}{p} \left(\frac{1}{p^w} - \frac{1}{p^{z_2}}\right)
\left(\frac{1}{p^w} - \frac{1}{p^{z_1}} \right) \right) F(a_{\ell}b_{\ell}, w, z_1, z_2) \eta(w, z_1, z_2).
\end{align*}Here $\eta(w, z_1, z_2)$ is an Euler product which is absolutely convergent in the region delimited by $\tRe w, \tRe z_1, \tRe z_2 > -1/2$ and we define 
\begin{align*}
F (a_{\ell}b_{\ell}, w, z_1, z_2) &= \prod_{p^j || a_{\ell}} \frac{1}{p^{(j - 1) w}}  \left(
   \frac{1}{p^w} - \frac{1}{p^{z_1}} \right)\left( 1 + \frac{1}{p^{1 + w}} \left( \frac{1}{p^w} - \frac{1}{p^{z_2}} \right) \right) \\
   &\prod_{p^j || b_{\ell}} \frac{1}{p^{(j - 1) w}}  \left(
   \frac{1}{p^w} - \frac{1}{p^{z_2}} \right)
   \left( 1 + \frac{1}{p^{1 + w}} \left( \frac{1}{p^w} - \frac{1}{p^{z_1}} \right) \right)\\
   &\prod_{p \nmid  a_{\ell}b_{\ell}}\left( 1+\frac{1}{p} \left(\frac{1}{p^w} - \frac{1}{p^{z_2}}\right)
\left(\frac{1}{p^w} - \frac{1}{p^{z_1}} \right) \right)^{-1}.
\end{align*}
Further, we may write

$$\prod_{p} \left( 1 + \frac{1}{p}  \left( \frac{1}{p^w} -
\frac{1}{p^{z_2}} \right)  \left( \frac{1}{p^w} - \frac{1}{p^{z_1}} \right)
\right) \eta ( w, z_1, z_2) = \frac{\zeta ( 1 + 2 w) \zeta ( 1 + z_1 +
z_2)}{\zeta ( 1 + w + z_1) \zeta ( 1 + w + z_2)} \tilde{\eta} ( w, z_1, z_2),$$
where $\tilde{\eta}$ denotes an Euler product which is absolutely convergent
in the region delimited by $\tRe w, \tRe z_1, \tRe z_2 > -
1 / 2$ and does not depend on $a_{\ell}$ or $b_{\ell}$. Thus,
\begin{multline*}
  S = \left( \frac{1}{2 \pi i} \right)^3 \left( \int_{( 2)} \right)^3
  \frac{\zeta ( 1 + 2 w) \zeta ( 1 + z_1 + z_2)}{\zeta ( 1 + w + z_!) \zeta (
  1 + w + z_2)} \tilde{\eta} ( w, z_1, z_2) 
  F (a_{\ell}b_{\ell}, w, z_1, z_2) \left( \frac{\alpha t + \beta}{2 \pi} \right)^w 
  \\ G ( w)
  \frac{dw}{w} \frac{N^{z_1} dz_1}{\log N z_{1^{}}^2}
  \frac{N^{z_2} dz_2}{\log N z_{2^{}}^2} 
\end{multline*}
and shifting contours to $\tRe w = - \delta$, $\tRe z_1 =
\tRe z_2 = \delta+\delta^2$ gives, since $\alpha t + \beta \asymp T$, $$
S = I_1 + I_2 + I_3 + O\bfrac{(a_{\ell}b_{\ell})^{\delta} N^{2\delta+2\delta^2}}{T^{\delta}}
$$
with $I_1, I_2, I_3$ specified below.
Since $N < T^{1/2 - \varepsilon}$ the error term is $\ll (a_{\ell}b_{\ell})^{\varepsilon}
T^{-\varepsilon}$ provided that $\delta$ is chosen small enough.
Writing $$
H(z_1, z_2) = \frac{\zeta ( 1 + z_1 + z_2)}{\zeta ( 1 + z_1) \zeta (
  1 + z_2)} \tilde{\eta} ( 0, z_1, z_2) F (a_{\ell}b_{\ell}, 0, z_1, z_2)
$$
we have 
$$
I_1 = \frac{\log (\alpha t + \beta)}{2} \frac{1}{(2 \pi i)^2} \int_{(1/4)} \int_{(1/4)} H(z_1, z_2) \cdot
\frac{N^{z_1} dz_1}{\log N z_{1}^2}
  \frac{N^{z_2} dz_2}{\log N z_{2}^2}, 
$$
\begin{align*}
I_2
= -\frac 12 \frac{1}{(2 \pi i)^2}  \int_{(1/4)} \int_{(1/4)}
  \bigg ( \frac{\zeta'}{\zeta}(1+z_1) + \frac{\zeta'}{\zeta}(1 + z_2) \bigg )
  \cdot H(z_1, z_2) \cdot \frac{N^{z_1} dz_1}{\log N z_{1^{}}^2}
  \frac{N^{z_2} dz_2}{\log N z_{2^{}}^2},
\end{align*}
and
\begin{multline*}
I_3
= \frac 12  \frac{1}{(2 \pi i)^2} \int_{(1/4)} \int_{(1/4)}
  \frac{\left(\frac{d}{dw}\tilde{\eta} ( w, z_1, z_2) F (a_{\ell}b_{\ell}, w, z_1, z_2)\right)_{w=0}}
  {\tilde{\eta} ( 0, z_1, z_2) F (a_{\ell}b_{\ell}, 0, z_1, z_2)}
  \cdot H(z_1,z_2) \cdot \\
\cdot \frac{N^{z_1} dz_1}{\log N z_{1^{}}^2}
  \frac{N^{z_2} dz_2}{\log N z_{2^{}}^2}.
\end{multline*}
Bounding the integrals is now a standard exercise.  As they can be bounded using the exact same procedure, we will focus our attention to $I_1$ (note in particular, that $I_3$ is smaller by a factor of $\log T$ compared with the other
integrals).

For ease of notation, write $G(z_1, z_2) = \tilde{\eta} ( 0, z_1, z_2) F (a_{\ell}b_{\ell}, 0, z_1, z_2)$.  Then 
\begin{multline*}
I_1
= \frac{\log (\alpha t + \beta)}{2} \sum_{n \leq N}\frac 1n \frac{1}{(2 \pi i)^2} \int_{(1/\log N)} \int_{(1/\log N)}\zeta ( 1 + z_1)^{-1} \zeta (
  1 + z_2)^{-1} 
  G(z_1, z_2) \cdot \\ \cdot \bfrac{N}{n}^{z_1+z_2}
  \frac{ dz_1}{\log N z_{1}^2}
  \frac{ dz_2}{\log N z_{2}^2}, 
\end{multline*}
Let $M = \exp (B(\log \log T)^2)$ for $B$ a parameter to be determined shortly.  We split the sum in $n$ above to $n\leq N/M$ and $n>N/M$.  

If $n>N/M$, then shift both contours to the line with real-part
$(\log M)^{-1}$ and bound the integrals trivially. The contribution of terms 
with $n > N / M$ is
$$\ll \log T (\log M)^5 (\log N)^{-2} (a_{\ell}b_{\ell})^{\epsilon} \ll \frac{(a_{\ell}b_{\ell})^\epsilon}{(\log T)^{1-\epsilon}}.$$

Now, for the terms with $n\leq N/M$, first truncate both contours at 
height $\log^4 T$ with an error $\ll (a_{\ell}b_{\ell})^\epsilon \cdot (\log T)^{-1}$.  Since $a_{\ell}b_{\ell}>1$, we assume without loss of generality that $a_{\ell}>1$.  This in turn implies that $F(a_{\ell}b_{\ell}, 0, 0, z_2) = 0$, so that the integrand is holomorphic at $z_1=0$.  From the classical zero free region for $\zeta(s)$, there exists a constant $c>0$ such that $(\zeta(1+z_1))^{-1} < \log(|z_1|+1)$ for $\tRe z_1 \geq 
-c (\log \log T)^{-1}$ and $|\tIm z_1| \leq \log^4 T$.  We now shift the integral in $z_1$ to $\tRe z_1 = -c (\log \log T)^{-1}$ with an error 
$\ll (a_{\ell}b_{\ell})^{\varepsilon} (\log T)^{-1}$ and bound the remaining integral trivially by
$$M^{\frac{-c}{\log \log T}} \log T \cdot (\log \log T)^2(a_{\ell}b_{\ell})^\epsilon  \ll \exp(-cB\log \log T) (\log T)^{1+\epsilon}(a_{\ell}b_{\ell})^\epsilon.
$$
The result follows upon picking $B = \frac{2}{c}$.
\end{proof}

\begin{proof}[Proof of Theorem \ref{thm:mollifiedmirage}]
Let $B(s) = M_{\theta}(s)$ with $0 < \theta < \tfrac 12$. 
Inserting the bound in Lemma \ref{lemma:saving} into 
Proposition \ref{prop:aver} we obtain
$$
\mathcal{E} \ll \frac{T}{(\log T)^{1 - \varepsilon}}
\cdot \sum_{\ell > 0} \frac{1}{(a_{\ell}b_{\ell})^{1/2 - \varepsilon}}  + O(T^{1 - \varepsilon}).
$$
The sum over $\ell > 0$ is rapidly convergent: Because of 
(\ref{condition4}) we have
$a_{\ell} \asymp b_{\ell} e^{2\pi \ell / \alpha}$ and therefore $a_{\ell}b_{\ell} \gg e^{2\pi \ell / \alpha}$.
It follows that the sum over $\ell > 0$ contributes $O(1)$ and we obtain
$
\mathcal{E} \ll T (\log T)^{-1 + \varepsilon}
$
as desired.
\end{proof}

\section{Proof of Theorem \ref{thm:nonvanishing}}
Recall that 
$$
M_{\theta}(s) := \sum_{n \geq 1} \frac{b(n)}{n^s}
$$
with coefficients 
$$
b(n) := \mu(n) \cdot \bigg ( 1 - \frac{\log n}{\log T^{\theta}}  \bigg ),
$$for $n\leq T^{\theta}$ and $b(n) = 0$ otherwise.
Define the mollified first moment as
\begin{equation}\label{eqn:I}
 \mathcal{I} := \sum_{\ell} \zeta(\tfrac 12 + i(\alpha \ell + \beta))
 M_{\theta} (\tfrac 12 + i(\alpha\ell+\beta))
 \phi
 \bfrac{\ell}{T},
\end{equation}and recall that
\begin{equation*}
\mathcal{J} := \sum_{\ell} |\zeta(\tfrac 12 + i(\alpha \ell + \beta))
 M_{\theta} (\tfrac 12+ i(\alpha\ell+\beta))|^2 \phi
\bfrac{\ell}{T}. 
\end{equation*}

By Cauchy-Schwarz and $0 \leq \phi \leq 1$, we have
$$
|\mathcal{I}| \leq (P_{\alpha,\beta}(T) \cdot T)^{1/2} \cdot \mathcal{J}^{1/2}.
$$
Then our Theorem \ref{thm:nonvanishing} follows from the following Proposition \ref{prop:I} and Theorem \ref{thm:mollifiedmirage}.
\begin{prop}\label{prop:I}
Let $\alpha > 0$, $\beta$ be real numbers.
With $\mathcal{I}$ as defined in (\ref{eqn:I}), and for $T$ large,
$$ |\mathcal{I}| = T\hat{\phi}(0) + O\bfrac{T}{\log T}.$$
\end{prop}
\begin{proof}
Uniformly in $0 \leq t \leq 2aT$ we have,
$$\zeta(\tfrac 12 + it) =  \sum_{n \leq 2aT} \frac{1}{n^{1/2 + it}} + O\bfrac{1}{T^{1/2}},
$$
Since in addition $|M(\tfrac 12 + it)| \ll T^{\theta/2+ \epsilon}$ for all $t$, 
we get
\begin{align*}
\mathcal{I}
&=  \sum_{\ell} \sum_{n \leq 2\alpha T} \frac{1}{n^{1/2 + i(\alpha \ell + \beta)}}
\cdot M (\tfrac 12 + i(\alpha\ell+\beta))
\phi\bfrac{\ell}{T} + O\big(T^{\theta/2 + 1/2 + \epsilon}\big)\\ 
&=  \sum_{m\leq T^{\theta}}\frac{b(m)}{\sqrt{m}} \sum_{n\leq 2\alpha T} \frac{1}{\sqrt{n}}
\cdot (mn)^{-ib} \sum_{\ell} (mn)^{-ia\ell} \phi\bfrac{\ell}{T} + O(T^{3/4})\\
&=  \sum_{m\leq T^{\theta}}\frac{b(m)}{\sqrt{m}} \sum_{n\leq 2\alpha T} \frac{1}{\sqrt{n}}
\cdot (mn)^{-ib} \sum_{\ell} T \hat{\phi}\left( T \left(\frac{\alpha \log (mn)}{2\pi} - \ell \right) \right) + O(T^{3/4}),
\end{align*}
by Poisson summation applied to the sum over $\ell$.

Note that $\hat{\phi}\left( Tc \right) \ll_A T^{-A}$ for any $|c|>T^{-1+\epsilon}$,which is an immediate result of $\hat{\phi}$ being a member of the Schwarz class. 
Hence, the sum above may be restricted to $|\ell| 
\leq \tfrac {\alpha}{2\pi} \cdot \log(2\alpha T^{1+\theta})+ O(T^{\epsilon - 1})$.  The terms with $\ell=0$ contributes a main term of $T \hat{\phi}(0)$ when $mn=1$, and the terms with other values of $mn$ contributes $O(T^{-A})$.  

Now consider $\ell \neq 0$.  Terms with $|a\log (mn) - \ell| > T^{\epsilon-1}$ contribute $O(T^{-A})$.  Otherwise, suppose that
\begin{equation}\label{eqn:prop1a}
\alpha = \frac {2\pi \ell}{\log n_0} + O(T^{\epsilon - 1}) 
\end{equation}
for some integer $n_0>1$, and fix such a $n_0$  The term $mn = n_0$ contributes
$$
\frac{T}{n_0^{1/2+ib}} \sum_{m|n_0} b(m)\hat{\phi}\left( T (\alpha \log (n_0) - \ell) \right) $$
for $T$ large. This term is bounded by
$$
\ll_{a} \frac{T}{\log T} \cdot \frac{d(n_0) \log n_0}{\sqrt{n_0}}
$$
because $b(m) = \mu(m) + O(\log m / \log T)$ for all $m$, and thus,
$$
\sum_{m | n_0}  b(m) \ll \frac{d(n_0) \log n_0}{\log T}
$$
For a fixed $\ell$, the number of $n_0$ satifying 
(\ref{eqn:prop1a}) is bounded by ${n_0} T^{-1 + \varepsilon} + 1$.  
Thus the total contribution of all the terms is
$$\ll T \frac{n_0^{1/2+\varepsilon}}{T} \cdot T^{\varepsilon} 
+ T \frac{d(n_0)}{\sqrt{n_0}} \frac{\log n_0}{\log T}
\ll T^{3/4+\varepsilon} + T \frac{d(n_0)}{\sqrt{n_0}} \frac{\log n_0}{\log T}. 
$$
We sum this over all the $|\ell| \ll \log (2\alpha T^{1 + \theta})$. 
Such a short sum does not affect the size of the first term above.
As for the second term, since $n_0 \asymp e^{2\pi a \ell}$, 
the sum over $\ell \neq 0$ is bounded by
$$\frac{T}{\log T} \sum_{|\ell| > 0} \frac{|\ell|}{e^{|a \pi \ell|(1-\epsilon)}} \ll 
\frac{T}{\log T}. 
$$
From this we have that
\begin{align*}
I
&=  \hat{\phi}(0)T +  O\bfrac{T}{\log T}
\end{align*}

\end{proof}

\begin{proof}[Proof of Theorem 1]
Appealing to a result of
Balasubramanian, Conrey and Heath-Brown \cite{BalasubramanianConreyHeathBrown} to compute
$\int_{T}^{2T} |(\zeta \cdot M_{\theta}) (\tfrac 12 + i(\alpha t + \beta))|^2 dt$, 
we have by Theorem \ref{thm:mollifiedmirage}
$\mathcal{J} \leq T \cdot (1 + \tfrac {1}{\theta} + o(1))$. Combining this with the inequality
$$
\mathcal{I} \leq (P_{\alpha,\beta}(T) \cdot T)^{1/2} \cdot \mathcal{J}^{1/2}
$$
and Proposition \ref{prop:I}, we obtain
$$
\hat{\phi}(0) T (1 + o(1)) \leq (P_{\alpha,\beta}(T) \cdot T)^{1/2} \cdot (T \cdot ( \tfrac{1}{\theta} + 1 + o(1)) )^{1/2}
$$
Hence,
$$
P_{\alpha,\beta}(T) \geq \frac{\theta}{\theta + 1} \hat{\phi}(0) + o(1)
$$
for all $0 < \theta < \tfrac 12$. Now we set $\phi(t) = 1$ for $t \in [1+\epsilon, 2-\epsilon]$ so that $\hat{\phi}(0) \geq 1-2\epsilon$.  Letting $\theta \rightarrow {\tfrac {1}{2}}^{-}$ and $\epsilon \rightarrow 0$,
we obtain the claim.
\end{proof}

In order to prove the Corollary we need the lemma below.

\begin{lemma}
We have,
$$
\sum_{\ell} |M_{\theta}(\tfrac 12 + i(\alpha \ell + \beta))|^2 \phi \big ( \frac{\ell}{T} \big )
 \ll T \log T
$$
\end{lemma}
\begin{proof}
Using Proposition 2 we find that the above second moment is equal to
$$
\int_{\mathbbm{R}} |M_{\theta}(\tfrac 12 + i(\alpha t + \beta))|^2 \phi \big ( \frac{t}{T} \big )
dt + O \bigg ( T \hat{\phi}(0) \sum_{\ell > 0} \frac{1}{\sqrt{a_{\ell}b_{\ell}}}
\cdot |F'(a_{\ell},b_{\ell})| \bigg )
$$
where $a_{\ell}, b_{\ell}$ denotes for each $\ell > 0$ the unique (if it exists!) couple
of co-prime integers such that $a_{\ell} b_{\ell} > 1$, $b_{\ell} < T^{1 / 2 - \varepsilon}
e^{-\pi \ell / \alpha}$ and
$$
\bigg | \frac{a_{\ell}}{b_{\ell}} - e^{2\pi \ell / \alpha} \bigg | \leq \frac{e^{2\pi \ell / \alpha}}{T^{1 - \varepsilon}}
$$
and where
$$
F'(a_{\ell},b_{\ell}) = \sum_{r \leq T} \frac{b(a_{\ell} r)b(b_{\ell} r)}{r} \ll \log T
$$
since the coefficients of $M_{\theta}$ are bounded by $1$ in absolute value. 
Since $\int_{\mathbbm{R}} |M_{\theta}(\tfrac 12 + i(\alpha t + \beta))|^2 \phi(t / T) dt \ll T \log T$
the claim follows.
\end{proof}
\begin{proof}[Proof of the Corollary]
Following \cite{IwaniecSarnak} let $\mathcal{H}_0$ be the set of
integers $T \leq \ell \leq 2T$ at which,
$$
|\zeta(\tfrac 12 + i(\alpha \ell + \beta))| \leq \varepsilon (\log \ell)^{-1/2}
$$
and let $\mathcal{H}_1$ be the set of integers $\ell$ at which the
reverse inequality holds. Notice that,
\begin{multline*}
\mathcal{C}_0 := \bigg | \sum_{\ell \in \mathcal{H}_0} \zeta(\tfrac 12 + i(\alpha \ell + \beta))
M_{\theta}(\tfrac 12 + i(\alpha \ell + \beta)) \phi \big ( \frac{\ell}{T} \big ) \bigg | \\
\leq \varepsilon (\log T)^{-1/2}T^{1/2} \cdot \bigg ( \sum_{\ell} |M_{\theta}(\tfrac 12 +
i(\alpha \ell + \beta))|^2 \phi \big ( \frac{\ell}{T} \big ) \bigg )^{1/2}
\leq C \varepsilon T \hat{\phi}(0)
\end{multline*}
for some absolute constant $C > 0$. Hence by Proposition \ref{prop:I} and the Triangle Inequality,
$$
\mathcal{C}_1 := \bigg | \sum_{\ell \in \mathcal{H}_1} \zeta(\tfrac 12 + i(\alpha \ell + \beta)) M_{\theta}(\tfrac 12 + i(\alpha \ell + \beta)) \phi \big ( \frac{\ell}{T} \big ) \bigg |
\geq (1 - C\varepsilon) \hat{\phi}(0) T
$$
while by Cauchy's inequality,
$$
\mathcal{C}_1 \leq \bigg ( \text{Card}(\mathcal{H}_1 ) \bigg )^{1/2}
\cdot \bigg ( \sum_{\ell} |\zeta(\tfrac 12 + i(\alpha \ell + \beta))M_{\theta}(\tfrac 12 + i(\alpha \ell + \beta))|^2 \phi \big ( \frac{\ell}{T} \big ) \bigg )^{1/2}
$$
As in the proof of Theorem 1, by Theorem 5 and a result of Balasubramanian, Conrey and Heath-Brown, the mollified second moment is $\leq T \cdot ( 1 + 1/\theta + o(1))$ as $T \rightarrow \infty$. Thus
$$
|\mathcal{H}_1| \geq \hat{\phi}(0) \frac{1 - C\varepsilon}{1 + 1/\theta} T
$$
Taking $\theta \rightarrow {\tfrac 12}^{-}$ and letting $\phi(t) = 1$
on $t \in [1 + \varepsilon; 2 - \varepsilon]$, so that $\hat{\phi}(0) \geq 1 - 2 \varepsilon$ we obtain the claim on taking $\varepsilon \rightarrow 0$. 
\end{proof}

\section{Large and small values: Proof of Theorem 5}

Let $0 \leq \phi \leq 1$ be a smooth function, compactly supported in $[1,2]$.
Let
$$
A(s) = \sum_{n \leq T} \frac{1}{n^s}
$$
and let
$$
B(s) = \sum_{n \leq N} b(n) n^{-s}
$$
be an arbitrary Dirichlet polynomial of length $N$.
Consider, 
$$
\mathcal{R} := \frac{\sum_{\ell} A(\tfrac 12 + i(\alpha \ell + \beta))|B(\tfrac 12 + i(\alpha \ell + \beta))|^2 
\phi\big ( \frac{\ell}{T} \big )} {\sum_{\ell} |B(\tfrac 12 + i(\alpha\ell + \beta))|^2 \phi \big ( \frac{\ell}{T} \big )}.
$$
Following Soundararajan \cite{Soundararajan}, and since $\zeta(\tfrac 12 + it) = A(\tfrac 12 + it) + O(t^{-1/2})$,
$$
\max_{T \leq \ell \leq 2T} |\zeta(\tfrac 12 + i(\alpha \ell + \beta))| + O(T^{-1/2})
\geq |\mathcal{R}|
\geq 
\min_{T \leq \ell \leq 2T} |\zeta(\tfrac 12 + i(\alpha \ell + \beta))| + O(T^{-1/2})
$$
Thus, to produce large and small values of $\zeta$ at discrete points $\tfrac 12 + i(a\ell+b)$
it suffices to choose a Dirichlet polynomial $B$ that respectively maximizes/minimizes the
ratio $\mathcal{R}$.
Fix $\varepsilon > 0$. 
Consider the set $S_1$ of tuples $(a_{\ell},b_{\ell})$, with $\ell \leq 2 \log T$,
such that
$$
\bigg | \frac{\alpha \log \frac{a_{\ell}}{b_{\ell}}}{2\pi} - \ell \bigg | \leq 
\frac{1}{T^{1 - \varepsilon}}
$$
and $a_{\ell}b_{\ell} > 1$ and both $a_{\ell},b_{\ell}$ are less than $T^{1/2 - \varepsilon}$.
In particular for each $\ell$ there is at most one such tuple so $|S_1| \leq 2 \log T$.
From each tuple in $S_1$ we pick one prime divisor of $a_{\ell}$ and one prime
divisor of $b_{\ell}$ and put them into a set we call $S$. 

We define our
resonator coefficients $r(n)$ by setting $L = \sqrt{\log N\log\log N}$ and
$$
r(p) = \frac{L}{\sqrt{p}\log p}
$$
when $p \in ([L^2; \exp((\log L)^2)]$ and $p \not \in S$. 
In the remaining cases we let $r(p) = 0$. Note in particular that
the resonator coefficients change with $T$. 

We then choose $b(n) = \sqrt{n}r(n)$ or $b(n) = \mu(n)\sqrt{n} r(n)$
depending on whether we want to maximize or minimize the ratio $\mathcal{R}$. 
For either choice of coefficients we have the following lemma.
\begin{lemma}
Write $D(s) = \sum_{n \leq T} \frac{a(n)}{n^s}$ with the coefficients $a(n) \ll 1$.
If $N = T^{1/2 - \delta}$ with $\delta > 10\varepsilon$, then,
\begin{multline*}
\sum_{\ell} D(\tfrac 12 + i(\alpha \ell + \beta)) |B(\tfrac 12 + i(\alpha \ell + \beta))|^2 \phi \big ( \frac{\ell}{T} \big )
= \int_{\mathbbm{R}} D(\tfrac 12 + i(\alpha t + \beta)) |B(\tfrac 12 + i(\alpha t + \beta))|^2 \phi \big ( \frac{t}{T}
\big ) dt \\ + O(T^{1 + (1 - 3\delta)/2+4\epsilon})
\end{multline*}
\end{lemma}
\begin{proof}

By Poisson summation we have,
\begin{multline*} \label{tobound}
\sum_{\ell} D(\tfrac 12 + i(\alpha \ell + \beta))|B(\tfrac 12 + i(\alpha \ell + \beta)|^2 \phi \big ( \frac{\ell}{T} \big ) = \\
= T \sum_{\ell} \sum_{\substack{m,n \leq N \\ h \leq T}} 
\frac{b(m)b(n)a(h)}{\sqrt{m n h}} \bigg ( \frac{m}{n h} \bigg )^{i\beta} 
\hat{\phi} \bigg ( T \bigg ( \frac{\alpha \log \frac{m}{n h}}{2\pi}
- \ell \bigg ) \bigg )
\end{multline*}
The term $\ell = 0$ contributes the main term (the continuous average). 
It remains to bound the remaining terms $\ell \neq 0$. 
Since $\hat{\phi}(x) \ll (1 + |x|)^{-A}$ the only surviving terms are those for which,
$$
\bigg | \frac{\alpha \log \frac{m}{n h}}{2\pi} - \ell \bigg | \leq \frac{1}{T^{1 - \varepsilon}}
$$
which in particular implies that $|\ell| \leq 2 \log T$.
We split our sum into two ranges, $nh < T^{1/2 - \epsilon}$ and $nh > T^{1/2 - \epsilon}$.  
\\
\textbf{First range}.
In the first range, for $(m, nh) = 1$, the real numbers $\log m/(nh)$
are spaced by at least $T^{-1 + \varepsilon}$ apart. 
Among all co-prime tuples with both $a_{\ell}, b_{\ell}$ less than $T^{1/2 - \varepsilon}$ there is at most one tuple satisfying,
$$
\bigg | \frac{\alpha \log \frac{a_{\ell}}{b_{\ell}}}{2\pi} - \ell \bigg | \leq \frac{1}{T^{1 - \varepsilon}}$$
Grouping the terms $m,n,h$ according to $m = a_{\ell} r$ and $n h= b_{\ell} r$, we re-write
the first sum sum over the range $ n h \leq T^{1/2 - \varepsilon}$ as follows,
$$
T  \sum_{\ell \neq 0} \frac{1}{\sqrt{a_{\ell}b_{\ell}}} 
\sum_{r} \frac{1}{r} \sum_{\substack{m,n  \leq N \\ n h \leq T^{1/2 - \varepsilon} \\ m = a_{\ell} r \\ nh = b_{\ell} r}} 
b(m)b(n)a(h) \bigg ( \frac{m}{n h} \bigg )^{i\beta} 
\hat{\phi} \bigg ( T \bigg ( \frac{\alpha \log \frac{m}{n h}}{2\pi}
- \ell \bigg ) \bigg )
$$
However by our choice of $r$ we have $b(a_{\ell}) = 0$, hence by multiplicativity
$b(m) = 0$, and it follows that the above sum is zero. 
\\
\textbf{Second range.} We now examine the second range $n h > T^{1/2 - \varepsilon}$. The condition $n h > T^{1/2 - \varepsilon}$ and $n \leq T^{1/2 - \delta}$
imply that $h > T^{\delta - \varepsilon}$. 
For fixed $m,n$ we see that there are at most $T^{\varepsilon}$ values of $h$ such that
$$
\bigg | \frac{\alpha \log \frac{m}{n h}}{2\pi} - \ell \bigg | \leq \frac{1}{T^{1-\varepsilon}}
$$
Putting this together we have
the following bound for the sum over $n h > T^{1/2 - \varepsilon}$, 
\begin{align*}
& T \left| \sum_{\ell \neq 0} \sum_{\substack{m,n \leq N, h \leq T \\ T^{1/2 - \varepsilon} < nh  }} 
\frac{b(m)b(n)a(h)}{\sqrt{m n h}} \bigg ( \frac{m}{n h} \bigg )^{i\beta} 
\hat{\phi} \bigg ( T \bigg ( \frac{\alpha \log \frac{m}{n h}}{2\pi}
- \ell \bigg ) \bigg ) \right| \\
&\ll 
T \sum_{|\ell| \leq 2 \log T}  \sum_{\substack{m, n \leq N}} \frac{|b(m)b(n)|}{\sqrt{m n}} \cdot T^{-\delta / 2 + \varepsilon} T^{\varepsilon}\\  
&\ll 
T^{1-\delta/2+3\varepsilon} \cdot N \sum_{n\leq N} \frac{|b(m)|^2}{n} 
\end{align*}
Then
$$
\sum_{n \leq N} \frac{|b(m)|^2}{m} \leq  \prod_{p \geq L^2} \bigg ( 1 + \frac{L^2}{p \log^2 p} \bigg ) \ll T^{\varepsilon}
$$
because $L^2 \sum_{p > L^2} p^{-1} (\log p)^{-2} \ll \log N / \log\log N
= o(\log T)$. Therefore the sum in the second range is bounded by $T^{1 - \delta / 2 + 4\varepsilon} N = T^{1 + (1 - 3\delta)/2 + 4\varepsilon}$. 
\end{proof}
In the above lemma we take $\delta = 1/3 + 4\varepsilon$,
so that $N = T^{1/6 - 4\varepsilon}$ and the error term is negligible
(that is $\ll T^{1 - \varepsilon}$).
Setting consecutively $D(s) = A(s)$ and $D(s) = 1$ we get,
$$
\mathcal{R} = \frac{\int_{\mathbbm{R}} A(\tfrac 12 + i(\alpha t + \beta)) |B(\tfrac 12 + i(\alpha t +\beta))|^2 \phi \big ( \frac{t}{T} \big ) dt}{\int_{\mathbbm{R}} |B(\tfrac 12 + i(\alpha t + \beta))|^2 \phi
\big ( \frac{t}{T} \big ) dt}
$$
plus a negligible error term. The above ratio was already worked out by Soundararajan
in \cite{Soundararajan} (see Theorem 2.1). Proceeding in the same way, we obtain that the above
ratio is equal to,
$$
\mathcal{R} = (1 + o(1)) \prod_{p} \bigg ( 1 + \frac{b(p)}{p} \bigg )
$$
Suppose that we were interested in small values, in which case $b(n) = \mu(n)\sqrt{n}r(n)$.
Then, 
$$
\mathcal{R} = (1 + o(1)) \prod_{p \not \in S} \bigg ( 1 - \frac{L}{p \log p} \bigg )
$$
Since 
\begin{equation*}
\sum_{p \in S} \frac{L}{p \log p} = \sum_{L^2 \leq p 
\leq L^2 + 2 \log T} \frac{L}{p \log p} = o \bigg ( \sqrt{\frac{\log N}{\log\log N}}
\bigg )
\end{equation*}
we find that
$$
\mathcal{R} = \exp \bigg ( - (1 +o(1)) \sqrt{\frac{\log N}{\log\log N}} \bigg )
$$
Recall that $N = T^{1/6 - 4\varepsilon}$. Letting $\varepsilon \rightarrow 0$
we obtain the claim since $\mathcal{R} \geq \min_{T \leq \ell \leq 2T} |\zeta(\tfrac 12
+ i(\alpha \ell + \beta))| + O(T^{-1/2})$. The large value estimate for the maximum 
of $\zeta(\tfrac 12 + i(\alpha \ell + \beta))$ is obtained in exactly
the same way by choosing $r(n) = \sqrt{n}r(n)$ instead. 

\section{Proof of the technical Proposition \ref{prop:aver}} 
Let $G(\cdot)$ be an entire function with rapid decay along vertical lines, 
that is $G(x + iy) \ll |y|^{-A}$ for any fixed $x$ and $A > 0$.
Suppose also that $G(-w) = G(w)$, $G(0) = 1$
and $\overline{G(w)} = G(\bar{w})$. An example of
such a function is $G(w) = e^{w^2}$. For such a function $G(x)$
we define a smooth function 
$$
W(x) := \frac{1}{2\pi} \int_{(\varepsilon)} x^{-w}  G(w) \cdot \frac{d w}{w}.
$$
Notice that $W$ is real.
\begin{lemma}[Approximate function equation]
We have, for $T < t < 2T$, 
  \[ | \zeta ( \tfrac{1}{2} + i t) |^2 = 2 \sum_{mn < T^{1 + \varepsilon}}
     \frac{1}{\sqrt{mn}} \cdot \bigg ( \frac{m}{n} \bigg )^{i t} 
     W \bigg ( \frac{2 \pi mn}{t}
\bigg ) + O (T^{- 2 / 3}) . \]
\end{lemma}
\begin{rem}
Of course we could work with the usual smoothing $V$ involving the 
Gamma factors on the Mellin transform side. We believe the 
smoothing $W(2\pi m n / t)$ to be (slightly) more transparent.
\end{rem}
\begin{proof}
By a standard argument (see \cite{IwaniecKowalski}, Theorem 5.3), 
\begin{equation} \label{standard}
|\zeta(\tfrac 12 + it)|^2 = \frac{2}{2\pi i} \int_{(\varepsilon)}
\zeta(\tfrac 12 + it + w)\zeta(\tfrac 12 - it + w)
\pi^{-w} G(w) \cdot 
g_{t}(w) 
\frac{dw}{w}.
\end{equation}
with $g_t(w) = \Gamma ( \tfrac 14 + \tfrac {it}{2} + \tfrac{w}{2} )
\Gamma ( \tfrac 14 - \tfrac{it}{2} + \tfrac{w}{2}) / \big (
\Gamma ( \tfrac 14 + \tfrac{it}{2} ) \Gamma (\tfrac 14 - \tfrac {it}{2}) \big )
$
By Stirling's formula $g_t(w) = (t/2)^{w} \cdot (1 + O((1 + |w|^2)/t))$
uniformly for $w$ lying in any fixed half-plane and $t$ large.
Using Weyl's subconvexity bound, 
on the line $\tRe w = \varepsilon$ we have
$\zeta(\tfrac 12 + it + w)\zeta(\tfrac 12 - it + w) \ll |t|^{1/3} + |w|^{1/3}$.
Therefore, the error term $O((1 + |w|^2)/t)$ 
in Stirling's approximation contributes
an error term of $O(T^{-2/3})$ in (\ref{standard}). Thus
$$
|\zeta(\tfrac 12 + it)|^2 = \frac{2}{2\pi i} \int_{(\varepsilon)} 
\zeta(\tfrac 12 + it + w)\zeta(\tfrac 12 - it + w)
\cdot \bigg ( \frac{t}{2\pi} \bigg )^{w} G(w) \cdot \frac{dw}{w}
+ O(T^{-2/3}).
$$
Shifting the line of integration to $\tRe w = 1 + \varepsilon$ we collect
a pole at $w = \tfrac 12 \pm it$, it is negligible because
$G(\tfrac 12 \pm it) \ll |t|^{-A}$. Expanding $\zeta(\tfrac 12 
+ it + w)\zeta(\tfrac 12 - it + w)$ into a Dirichlet series on the line
$\tRe w = 1 + \varepsilon$ we conclude that
$$
|\zeta(\tfrac 12 + it)|^2 = 2\sum_{m , n \geq 1} \frac{1}{\sqrt{m n}}
\cdot \bigg ( \frac{m}{n} \bigg )^{it} W \bigg ( \frac{2\pi m n}{t} \bigg )
+ O(T^{-2/3}).
$$
Notice that $W(x) = O_A(x^{-A})$ for $x > 1$. Since $T \leq t \leq 2T$
if $m n > T^{1 + \varepsilon}$ then $2\pi m n / t \gg T^{\varepsilon}$. Therefore
we can truncate
the terms with $m n > T^{1 + \varepsilon}$ making an error term of
at most $\ll T^{-A}$. The claim follows.
\end{proof}

%
%
Recall also that
\[ B (s) \assign \sum_{n \leqslant T^{\theta}} \frac{b (n)}{n^s} \]
Therefore, 
\begin{eqnarray} \label{sex}
 \nonumber \mathcal{J} & \assign & \sum_{\ell \in \mathbbm{Z}} | \zeta ( \tfrac{1}{2} +
  i (\alpha \ell + \beta)) B ( \tfrac{1}{2} + i (\alpha \ell + \beta)) |^2 
  \cdot \phi
  ( \frac{\ell}{T})\\ \nonumber
  & = & 2\sum_{m n < T^{1 + \varepsilon}} \frac{1}{\sqrt{mn}} \sum_{h, k \leqslant
  T^{\theta}} \frac{b (h) b (k)}{\sqrt{hk}} \sum_{\ell \in \mathbbm{Z}} \bigg (
  \frac{mh}{nk} \bigg )^{i (\alpha \ell + \beta)} W \bigg 
  ( \frac{2 \pi mn}{\alpha \ell + \beta}
\bigg ) \phi \big (
  \frac{\ell}{T} \big ) + O \big (T^{5/6 + \varepsilon} \big ) \\  
 & = & 2 \sum_{m n < T^{1 + \varepsilon}} \frac{1}{\sqrt{mn}} \sum_{h, k \leqslant
  T^{\theta}} \frac{b (h) b (k)}{\sqrt{hk}} 
  \cdot \bigg ( \frac{ m h}{n k} \bigg )^{i \beta } \sum_{\ell \in \mathbbm{Z}} 
 \hat{f}_{m, n, T} \bigg ( \frac{\alpha \log
     \tfrac{mh}{nk}}{2 \pi} - \ell \bigg ) + O(T^{5/6 + \varepsilon})
%
\end{eqnarray}
using Poisson summation in the sum over $\ell$, with
$$
f_{m,n,T}(x) :=  W \bigg ( \frac{2 \pi mn}{\alpha x + \beta
} \bigg ) \cdot \phi \big 
( \frac{x}{T} \big )
$$
\subsection{The main term $\ell = 0$}
Consider the sum with $\ell = 0$,
\begin{align*}
& 
2\sum_{m n < T^{1 + \varepsilon}} \frac{1}{\sqrt{mn}} \sum_{h, k \leqslant
  T^{\theta}} \frac{b (h) b (k)}{\sqrt{hk}}  
 \cdot \bigg ( \frac{m k}{n h} \bigg )^{i\beta} \cdot \hat{f}_{m,n,T} \bigg ( \frac{\alpha \log \frac{mk}{nh}}{2\pi} \bigg )  \\
& = 2\sum_{m n < T^{1 + \varepsilon}} \frac{1}{\sqrt{m n}} \sum_{h , k \leq T^{\theta}}
 \frac{b(h) b(k)}{\sqrt{h k}} \cdot \int_{\mathbbm{R}} \bigg ( \frac{m k}{n h}
\bigg )^{i (\alpha t + \beta)} W \bigg ( \frac{2\pi m n}{\alpha t + \beta}
\bigg ) \phi \bigg ( 
\frac{t}{T} \bigg ) dt
\end{align*}

Interchanging the sums and the integral, this becomes
\begin{equation}
\label{desired}
\int_{\mathbbm{R}} |B(\tfrac 12 + i (\alpha t + \beta))|^2 \cdot
2 \sum_{m n < T^{1 + \varepsilon}} \frac{1}{\sqrt{m n}} \cdot \bigg ( \frac{m}{n}
\bigg )^{i (\alpha t + \beta) } W \bigg ( \frac{2\pi m n}{\alpha t + \beta}
\bigg ) \phi \big ( \frac{t}{T} \big ) dt
\end{equation}
By the approximate functional equation,
$$
2 \sum_{m n < T^{1 + \varepsilon}} \frac{1}{\sqrt{mn}} \bigg ( \frac{m}{n} \bigg )^{i ( \alpha t + \beta)} W \bigg ( \frac{2\pi m n}{\alpha t + \beta} \bigg ) = 
|\zeta(\tfrac 12 + i( \alpha t
+ \beta)|^2 + O(T^{-2/3}).
$$
Therefore (\ref{desired}) is
$$
\int_{\mathbbm{R}} |B(\tfrac 12 + i(\alpha t + \beta))
\zeta(\tfrac 12 + i(\alpha t + \beta))|^2 \phi 
\bigg ( \frac{t}{T} \bigg ) dt + O(T^{1 - \varepsilon})
$$
as desired. 

\subsection{The terms $\ell \neq 0$}
Since $$
\hat{f}_{m,n,T} \bigg ( \frac{\alpha \log \frac{m h}{n k} }{2\pi} + \ell \bigg )
= \overline{\hat{f}_{m,n,T} \bigg ( \frac{\alpha 
    \log \frac{n k}{m h}}{2\pi} - \ell
\bigg )}
$$
we can re-write the sum over $\ell \neq 0$ so as to have $\ell > 0$
in the summation,
$$
\mathcal{J}_{0} = 2 \sum_{\ell > 0} \sum_{\substack{m n < T^{1 + \varepsilon}
\\ h, k \leq T^{\theta}
}} \frac{b(h)b(k)}{\sqrt{m n h k}}
\cdot 2 \tRe \bigg ( \bigg (\frac{ m h}{n k} \bigg )^{i \beta} \hat{f}_{m,n,T}
\bigg ( \frac{\alpha \log \frac{m h}{n k}}{2\pi} - \ell \bigg ) \bigg )
$$
Differentiating repeatedly and using that $W$ and all derivatives of $W$ are Schwarz class, 
we find that for $mn < T^{1 + \varepsilon}$,
$
f^{(k)}_{m,n,T}(x) \ll T^{-k}
$
for all $x$. Therefore for any fixed $A > 0$,
$$
\hat{f}_{m,n,T}(x) \ll_{A} T \big (1 + T |x| \big )^{-A}
$$
It follows that the only integers $m,n,k,h,\ell$ that contribute to
$\mathcal{J}_0$ are the $m,n,k,h,\ell$ for which
$$
\bigg | \frac{\alpha}{2\pi} \cdot \log \frac{m h}{n k} - \ell \bigg | \leq
T^{-1  + \eta}.
$$
for some small, but arbitrary $\eta > 0$. 
This condition implies that
\begin{equation} \label{condition}
\bigg | \frac{m h}{n k} - e^{2\pi \ell / \alpha} \bigg | \leq e^{2\pi\ell/ \alpha}
T^{-1 + \eta}
\end{equation}
and we might as-well restrict the sum in $\mathcal{J}_0$ to those $m,n,k,h,\ell$
satisfying this weaker, but friendlier, condition. Thus,
\begin{equation} \label{Mainequation}
\mathcal{J}_{0} = 4 \tRe \sum_{\ell > 0}
\sum_{\substack{m n < T^{1 + \varepsilon} \\ h,k \leq T^{\theta}
\\ m,n,h,k \text{ satisfy } (\ref{condition})}} \frac{b(h)b(k)}{\sqrt{m n h k}} 
\cdot \bigg ( \frac{m k}{n h} \bigg )^{i \beta}
\hat{f}_{m,n,T} \bigg ( \frac{\alpha \log \frac{m k}{n h}}{2\pi} 
 - \ell \bigg ) + O_A\bfrac{1}{T^A}.
\end{equation}
Now for a fixed $\ell > 0$, consider the inner sum over $m, n, h, k$ in (\ref{Mainequation}).  We group together terms in the following way: If the integres $m, n, k, h$
satisfy (\ref{condition}) then we let $a_{\ell} = mk / (mk, nh)$ and $b_{\ell} 
= nh
/ (mk, nh)$ so that $(a_{\ell}, b_{\ell}) = 1$.  
We group together all multiples of $a_{\ell}, b_{\ell}$ of the form $mk = a_{\ell} r$ and $nh = b_{\ell} r$ with a common $r > 0$. The $a_{\ell}, b_{\ell}$ 
are co-prime and satisfy
\begin{equation}
  \bigg | \frac{a_{\ell}}{b_{\ell}} - e^{2 \pi \ell / \alpha} \bigg | < 
\frac{e^{2 \pi \ell / \alpha}}{T^{1 -
  \eta}} . \label{pq}
\end{equation}
This allows us to write
\begin{equation}
  \mathcal{J}_{0} = 4 \tRe \sum_{\ell > 0} \sum_{\substack{
    a_{\ell}, b_{\ell} \geqslant 1\\
    (a_{\ell}, b_{\ell}) = 1\\
    \text{satisfy } ( \ref{pq})
  }} \sum_{r \geqslant 1} \sum_{\substack{
    mn \leqslant T^{1 + \varepsilon}\\
    h, k \leqslant T^{\theta}\\
    nh = a_{\ell} r\\
    mk = b_{\ell} r
  }} \frac{b (h) b (k)}{\sqrt{mknh}} \cdot \bigg (
  \frac{a_{\ell}}{b_{\ell}} \bigg )^{\mathi \beta} \cdot \hat{f}_{m, n, T} \bigg (
  \frac{\alpha \log (a_{\ell} /
  b_{\ell})}{2 \pi} - \ell \bigg ) . \label{fourtherquation}
\end{equation}
It is useful to have a bound for the size of $b_{\ell}$ in the above sum. Equation
(\ref{pq}) implies that $a_{\ell} \asymp b_{\ell} \cdot e^{2 \pi \ell / \alpha}$.
Furthermore, since $mn < T^{1 + \varepsilon}$, $h, k \leqslant T^{\theta}$
and $a_{\ell} r = mk$, $b_{\ell} r = nh$ we have $a_{\ell} \cdot
b_{\ell} < mnkh < T^{1 + 2\theta + \varepsilon}$. Combining $a_{\ell} \asymp b_{\ell}
\cdot 
e^{2 \pi \ell / \alpha}$ and $a_{\ell} b_{\ell} < T^{1 +2\theta + \varepsilon}$ we obtain $b_{\ell} < T^{1 / 2 + \theta + \varepsilon} \cdot e^{- \pi \ell / \alpha}$. 
Let
\begin{align*}
K_{\ell} & := T^{1/2 - \eta} e^{-\pi \ell / \alpha} \\
M_{\ell} & := T^{1/2 + \theta + \varepsilon} e^{-\pi\ell/\alpha}
\end{align*}
We split the sum according to
whether $b_{\ell} < K_{\ell}$ or $b_{\ell} > K_{\ell}$, getting 
\[ \mathcal{J}_{0}= 4 \tRe \sum_{\ell > 0} \sum_{\substack{
     b_{\ell} < M_{\ell}\\
     a_{\ell} \geqslant 1\\
     (a_{\ell}, b_{\ell}) = 1\\
     \text{satisfy } ( \ref{pq})
   }} \frac{(a_{\ell} / b_{\ell})^{\mathi \beta}}{\sqrt{a_{\ell} b_{\ell}}} \sum_{r
   \geqslant 1} \frac{1}{r} \sum_{\substack{
     mn \leqslant T^{1 + \varepsilon}\\
     h, k \leqslant T^{\theta}\\
     nh = b_{\ell} r\\
     mk = a_{\ell} r
   }} \hat{f}_{m, n, T} \bigg (
 \frac{\alpha \log (a_{\ell} / b_{\ell})}{2 \pi} - \ell
\bigg )
   = 4\tRe (S_1 + S_2) \]
where $S_1$ is the sum over $b_{\ell} \leq K_{\ell}$ and $S_2$ is the corresponding
sum over $M_{\ell} > 
b_{\ell} > K_{\ell}$. 
To finish the proof of the Proposition it remains to
evaluate $S_1$ and $S_2$. The sum $S_1$ can give a main term contribution in the context of Theorem \ref{thm:mirage} depending on the Diophantine properties of $a$, while bounding $S_1$ as an error term in the context of Theorem \ref{thm:nonvanishing} is relatively subtle.  In contrast, $S_2$ is always negligible.  

We first furnish the following expression for $S_1$.
\begin{lemma}
For each $\ell > 0$ there is at most one tuple of co-prime integers 
$(a_{\ell},b_{\ell})$ such that $a_{\ell}b_{\ell} > 1$ , $b_{\ell} < K_{\ell}
= T^{1/2 - \eta} e^{-\pi \ell / \alpha}$ and
such that
\begin{equation} \label{condition2}
  \bigg | \frac{a_{\ell}}{b_{\ell}} - e^{2\pi \ell / \alpha} \bigg | \leq
  \frac{e^{2\pi \ell / \alpha}}{T^{1 - \eta}}.
\end{equation}
We denote by $\sum_{\ell}^{*}$ the sum over $\ell$'s satisfying the
above condition. Then,
  $$ S_1 = T  \cdot {\sum_{\ell > 0}}^{*}
  \frac{(a_{\ell}/b_{\ell})^{i\beta}}{\sqrt{a_{\ell} b_{\ell}}} 
  \int_{- \infty}^{\infty} \phi \bigg (
  \frac{t}{T} \bigg )
  \cdot \exp \bigg ( - 2\pi i t \bigg (
  \frac{\alpha \log  \frac{ a_{\ell}}{b_{\ell}}}{2\pi} - \ell \bigg ) \bigg )
  \cdot F(a_{\ell},b_{\ell}, t) d t
  $$
where
\begin{align*}
F (a_{\ell},b_{\ell}, t) & := \sum_{h,k \leq T^{\theta}} b(h) b(k)
\sum_{r \geq 1} \frac{1}{r} \sum_{\substack{m,n \geq 1 \\ mk = a_{\ell} r\\
nh = b_{\ell} r}} W \bigg ( \frac{\alpha t + \beta}{2\pi m n} \bigg )
\\
& = \sum_{m,n \leq T^{\theta}} \frac{b(m)b(n)}{m n}
\cdot (m a_{\ell}, n b_{\ell}) \cdot 
 \mathcal{H} \bigg ( (\alpha t + \beta) \cdot 
\frac{(ma_{\ell},nb_{\ell})^2}{2\pi m a_{\ell} n b_{\ell}} \bigg )
\end{align*}
and
$$
\mathcal{H}(x) = \frac{1}{2\pi i}\int_{(\varepsilon)} 
\zeta(1 + 2w) \cdot x^w G(w) \cdot \frac{dw}{w}
 = 
\begin{cases}
  \tfrac 12 \cdot \log x + \gamma  + O_{A}(x^{-A}) & \text{ if } x \gg 1
\\
 O_A(x^{A}) & \text{ if } x \ll 1
\end{cases}
$$
\end{lemma}

\begin{Proof}
  Given $\ell$, there is at most one $b_{\ell} \leqslant K_{\ell}$ 
  for which there is a
  co-prime $a_{\ell}$ such that $( \ref{condition2})$ 
  holds, because Farey fractions with
  denominator $< K_{\ell}$ are spaced at least $K_{\ell}^{-2} = e^{2\pi \ell / \alpha}
  T^{-1 + 2\eta}$
  far apart. 
  Thus for each $\ell$, the
  sum over $a_{\ell}, b_{\ell}$ in
  $S_1$ consists of at most one element $(a_{\ell}, b_{\ell})$,
  \[ S_1 = {\sum_{\ell > 0}}^* \frac{(a_{\ell} / b_{\ell})^{\mathi
      b}}{\sqrt{a_{\ell} b_{\ell}}} \sum_{r \geqslant 1} \frac{1}{r}
     \sum_{\substack{
       mn \leqslant T^{1 + \varepsilon}\\
       h, k \leqslant T^{\theta}\\
       nh = b_{\ell} r\\
       mk = a_{\ell} r
     }} \frac{b (h) b (k)}{\sqrt{mknh}} \cdot 
     \hat{f}_{m, n, T} \bigg ( \frac{\alpha \log (a_{\ell} /
     b_{\ell})}{2 \pi} - \ell \bigg ) \]
  To simplify the above expression we write
  \begin{eqnarray*}
    \hat{f}_{m, n, T} (x) & = & \int_{- \infty}^{\infty} W \bigg ( \frac{2 \pi
    mn}{\alpha t + \beta} \bigg ) 
    \phi \bigg (\frac{t}{T} \bigg ) e^{- 2 \pi \mathi xt} \mathd t
  \end{eqnarray*}
  The sum $S_1$ can be now re-written as,
  \begin{multline*}
  T {\sum_{\ell > 0}}^* \frac{(a_{\ell}/b_{\ell})^{ib}}{\sqrt{a_{\ell}b_{\ell}}}
  \int_{- \infty}^{\infty} \phi \bigg (\frac{t}{T} \bigg ) 
  \exp \bigg ( - 2\pi i
  t \bigg ( \frac{\alpha \log \frac{m h}{n k}}{2\pi} - \ell \bigg )\bigg )
  \cdot \\
     \sum_{r \geqslant 1}
     \frac{1}{r} \sum_{h, k \leqslant T^{\theta}} b (h) b (k)
     \sum_{\substack{
       mn < T^{1 + \varepsilon}\\
       nh = b_{\ell} r, mk = a_{\ell} r
     }} W \bigg ( \frac{2 \pi mn}{\alpha t + \beta} 
     \bigg ) \mathd t. \end{multline*}
  Since $W (x) \ll x^{- A}$ for $x > 1$ and $at +b \asymp T$ we complete the sum
  over $mn < T^{1 + \varepsilon}$ to $m, n \geqslant 1$ making a negligible
  error term $\ll_A T^{- A}$. To finish the proof it remains to
  understand the expression
 \begin{equation} \label{remains}
  \sum_{h ,k \leq T^{\theta}} b(h)b(k)
  \sum_{r} \frac{1}{r} 
  \sum_{\substack{m, n \geq 1 \\ m k = a_{\ell} r \\ nh = b_{\ell} r}}
  W \bigg ( \frac{2\pi m n}{\alpha t + \beta} \bigg )
\end{equation}
  We notice that
\begin{equation} \label{almost}
  \sum_{r} \frac{1}{r} \sum_{\substack{m,n \geq 1 \\
 n h = b_{\ell} r \\ m k = a_{\ell} r}} W\bigg ( \frac{2\pi m n}{\alpha t + \beta}
  \bigg ) = \frac{1}{2\pi} \int_{(\varepsilon)}
  \sum_{r} \frac{1}{r}\sum_{\substack{m,n \geq 1 \\
      n h = b_{\ell} r \\ m k = a_{\ell} r}} \frac{1}{(m n)^{w}}
\cdot
  \bigg ( \frac{\alpha t + \beta}{2\pi} \bigg )^{w} G(w) \frac{d w}{w}
\end{equation}
Furthermore $nh = b_{\ell} r$ and $mk = a_{\ell} r$
imply that $mk b_{\ell} = nh a_{\ell}$. On the other hand
since $a_{\ell}$ and $b_{\ell}$ are co-prime
the equality $mk b_{\ell} = nh a_{\ell}$ implies that
there exists a unique $r$ such that $nh = b_{\ell} r$
and $mk = a_{\ell} r$. We notice as-well that this
unique $r$ can be expressed as $((a_{\ell} b_{\ell})/(m k n h))^{-1/2}$. 
Therefore we have the equality,
$$
 \sum_{r} \frac{1}{r}\sum_{\substack{m,n \geq 1 \\
      n h = b_{\ell} r \\ m k = a_{\ell} r}} \frac{1}{(m n)^{w}}
= \sum_{\substack{m,n \geq 1 \\ n h a_{\ell} = mk b_{\ell}}}
\frac{1}{(mn)^w} \cdot \sqrt{ \frac{a_{\ell} b_{\ell}}{m k n h} }
$$
We express the condition $n h a_{\ell} = mk b_{\ell}$
as $h a_{\ell} | k b_{\ell} m$ and $n = k b_{\ell} m / (h a_{\ell})$ so
as to reduce the double sum over $m,n$ to a single
sum over $m$. Furthermore the condition
$h a_{\ell} | k b_{\ell} m$ can be dealt with by noticing
that it is equivalent to $h a_{\ell} / (h a_{\ell}, k b_{\ell}) | m$.
Using these observations we find that,
\begin{equation*}
 \sum_{\substack{m,n \geq 1 \\ n h a_{\ell} = mk b_{\ell}}}
\frac{1}{(mn)^w} \cdot \sqrt{ \frac{a_{\ell} b_{\ell}}{m k n h} }
 = \frac{(h a_{\ell},k b_{\ell})}{h k} \cdot \zeta(1 + 2w) \cdot \bigg ( 
\frac{ (h a_{\ell}, k b_{\ell})^2}{h a_{\ell} k b_{\ell}} \bigg )^{w} 
\end{equation*}
Plugging the above equation into (\ref{almost}) it follows that
%
\begin{equation*}
  \sum_{r \geq 1} \frac{1}{r} \sum_{\substack{m ,n \geq 1 \\
      m k = b_{\ell} r \\ n h = a_{\ell} r}}
  W \bigg ( \frac{2\pi m n}{\alpha t + \beta} \bigg )
= \frac{(h a_{\ell}, k b_{\ell})}{h k} \cdot \mathcal{H} \bigg ( \frac{\alpha t + \beta}{2\pi m n} \bigg )
\end{equation*}
An easy calculation reveals that $\mathcal{H}(x) = (1/2)\log x + \gamma + O_A(x^{-A})$ for $x \gg 1$ and that $\mathcal{H}(x) = O_A(x^A)$ for $x \ll 1$.
We conclude that equation (\ref{remains}) equals to
\begin{equation*}
  \sum_{h , k \leq T^{\theta}} 
\frac{b(h)b(k)}{h k} \cdot (h a_{\ell}, k b_{\ell}) \cdot \mathcal{H} \bigg
 ( \frac{\alpha t + \beta}{2\pi m n} \bigg )
\end{equation*}
as desired. 
\end{Proof}

The second sum $S_2$ can be bounded directly.

\begin{lemma}
  We have $S_2 \ll T^{1 / 2 + \theta + \varepsilon}$.
\end{lemma}

\begin{Proof}
  Recall that the $a_{\ell}, b_{\ell}$ are always assumed to satisfy the condition
  \begin{equation} \label{condition3}
    \bigg | \frac{a_{\ell}}{b_{\ell}} - e^{2\pi \ell / \alpha} \bigg | \leq \frac{e^{2\pi \ell / \alpha}}
    {T^{1 - \eta}}.
  \end{equation}
  Recall also that
  \begin{align*}
    K_{\ell} & := T^{1/2 - \eta} e^{-\pi \ell / \alpha} \\
    M_{\ell} & := T^{1/2 + \theta + \varepsilon} e^{-\pi \ell / \alpha}
  \end{align*}
  Then,
  \begin{equation}
    S_2 = \sum_{\ell \in \mathbbm{Z}} \sum_{\substack{
      K_{\ell} < b_{\ell} < M_{\ell}\\
      a_{\ell} \geqslant 1\\
      (a_{\ell}, b_{\ell}) = 1\\
      \text{satisfy } ( \ref{condition3})
    }} \frac{(a_{\ell} / b_{\ell})^{\mathi \beta}}{\sqrt{a_{\ell} b_{\ell}}} \sum_{r
    \geqslant 1} \frac{1}{r} \sum_{h, k \leqslant T^{\theta}} b (h) b (k)
    \sum_{\substack{
      mn \leqslant T^{1 + \varepsilon}\\
      nh = b_{\ell} r\\
      mk = a_{\ell} r
    }} \hat{f}_{m, n, T} \bigg ( \frac{\alpha \log (a_{\ell} / b_{\ell})}
        {2 \pi} - \ell
    \bigg )
    . \label{fourtherquation}
  \end{equation}
  We split the above sum into dyadic blocks $b_{\ell} \asymp N$ 
  with $K_{\ell} < N <
  M_{\ell}$. The
  number of $(a_{\ell}, b_{\ell}) = 1$ with $b_{\ell} \asymp N$ and satisfying
  (\ref{condition3}) is bounded by
  \[ \ll \frac{e^{2 \pi \ell / \alpha}}{T^{1 - \eta}} \cdot N^2 + 1 \]
  because Farey fractions with denominators of size $\asymp N$ are spaced at
  least $N^{- 2}$ apart. Therefore, for a fixed $\ell$, using the bounds $b (n) \ll n^{\varepsilon}$  and $\hat{f}_{m, n, T} (x) \ll T$,  the dyadic block with
  $b_{\ell} \asymp N$ contributes at most,
  \begin{equation}
    \label{tobound}
    \ll T^{1 + \varepsilon} \sum_{\substack{
       b_{\ell} \asymp N\\
       a_{\ell} \geqslant 1\\
       (a_{\ell}, b_{\ell}) = 1\\
       a_{\ell}, b_{\ell} \text{satisfy } ( \ref{condition3} )
     }} \frac{1}{(a_{\ell} b_{\ell})^{1/2}} \sum_{r < T^2} \frac{1}{r}
     \sum_{\substack{
       m, n, h, k\\
       mk = a_{\ell} r\\
       nh = b_{\ell} r
     }} 1 \ll T^{1 + \varepsilon} \sum_{\substack{
       b_{\ell} \asymp N\\
       a_{\ell} \geqslant 1\\
       (a_{\ell}, b_{\ell}) = 1\\
      a_{\ell}, b_{\ell} \text{ satisfy } ( \ref{condition3})
     }} \frac{1}{(a_{\ell} b_{\ell})^{1/2 - \varepsilon}} 
       \end{equation}
because
$$
\sum_{r \leq T^{2}} \frac{1}{r} \sum_{\substack{m,h,k,n \\
      m h = a_{\ell} r \\ n k = b_{\ell} r}} 1 = 
\sum_{r \leq T^{2}} \frac{d(a_{\ell} r)d(b_{\ell} r)}{r} \ll (T a_{\ell} b_{\ell})^{\varepsilon}. 
$$
Since $a_{\ell} \asymp b_{\ell} \cdot e^{2 \pi \ell / \alpha}$ the sum (\ref{tobound}) is bounded by,
  \[ \ll \frac{T}{N} \cdot (T N)^{\varepsilon} e^{- (1 - \varepsilon) 
    \pi \ell / \alpha} \cdot \left ( \frac{e^{2 \pi \ell
     / \alpha}}{T^{1 - \eta}} \cdot N^2 + 1 \right  ). \]
  Keeping $\ell$
  fixed and summing over all possible dyadic blocks $ K_{\ell} < N < M_{\ell} $ 
  shows that for fixed $\ell$ the inner
  sum in (\ref{fourtherquation}) is bounded by
  \begin{align} \label{fixedbound}
   & \ll T^{\varepsilon+ \eta} \cdot e^{\pi ( 1 + \varepsilon) \ell / \alpha}
    \cdot M_{\ell}^{1 + \varepsilon} + T^{1 + \varepsilon}
    \cdot K_{\ell}^{-  1 + \varepsilon} \cdot e^{-(1 -\varepsilon)
      \pi \ell / \alpha} \\ \nonumber
   & \ll T^{1/2 + \theta + \varepsilon + \eta} \cdot e^{\varepsilon \ell / \alpha} + 
    T^{1 /2 + \eta + \varepsilon} \cdot e^{\varepsilon \ell / \alpha}.
\end{align}
The condition (\ref{condition3}) restricts $\ell$ to 
$0 < \ell < 2 \alpha \log T$. 
Summing (\ref{fixedbound}) over all $0 < \ell < 2 \alpha \log T$
we find that $S_2$ is bounded by
$
T^{1/2 + \theta + 2\varepsilon + \eta} + T^{1/2 + \eta + \varepsilon}
$. Since $\theta < \tfrac 12$ and we can take $\eta, \varepsilon$ 
arbitrarily small, but fixed,
the claim follows.
\end{Proof}

\section{Acknowledgements}
This work was done while both authors were visiting Centre de Recherches Math\'ematiques.  We are grateful for their kind hospitality.

\bibliography{draft}
\bibliographystyle{plain}





\end{document}